\documentclass[12pt]{amsart}
\usepackage{amsmath,amssymb,amsthm,amsrefs}
\usepackage{enumitem}
\usepackage{color}
\theoremstyle{plain}
\newtheorem{theorem}{Theorem}[section]

\newtheorem{corollary}[theorem]{Corollary}
\newtheorem{lemma}[theorem]{Lemma}
\theoremstyle{definition}

\newtheorem{example}[theorem]{Example}

\newtheorem*{problem}{Inverse Spectral Problem}

\makeatletter
\@addtoreset{equation}{section}
\makeatother

\makeatletter
\newcommand{\Spvek}[2][r]{%
  \gdef\@VORNE{1}
  \left(\hskip-\arraycolsep%
    \begin{array}{#1}\vekSp@lten{#2}\end{array}%
  \hskip-\arraycolsep\right)}

\def\vekSp@lten#1{\xvekSp@lten#1;vekL@stLine;}
\def\vekL@stLine{vekL@stLine}
\def\xvekSp@lten#1;{\def\temp{#1}%
  \ifx\temp\vekL@stLine
  \else
    \ifnum\@VORNE=1\gdef\@VORNE{0}
    \else\@arraycr\fi%
    #1%
    \expandafter\xvekSp@lten
  \fi}
\makeatother
\numberwithin{equation}{section}

\begin{document}
\title[Ambarzumian-type problems for discrete Schr\"{o}dinger operators]
{Ambarzumian-type problems for discrete Schr\"{o}dinger operators\\}

\author[Eakins]{Jerik Eakins}\address{Department of Mathematics, Texas A{\&}M University, College Station, TX 77843, U.S.A.}\email{jerikeakins@tamu.edu}
\author[Frendreiss]{William Frendreiss}\address{Department of Mathematics, Texas A{\&}M University, College Station, TX 77843, U.S.A.}\email{wfrendreiss@tamu.edu}
\author[Hat\.{i}no\u{g}lu]{Burak Hat\.{i}no\u{g}lu}\address{Department of Mathematics, UC Santa Cruz, Santa Cruz, CA 95064, U.S.A.}\email{bhatinog@ucsc.edu}
\author[Lamb]{Lucille Lamb}\address{Department of Mathematics, Texas A{\&}M University, College Station, TX 77843, U.S.A.}\email{lucille.lamb@tamu.edu}
\author[Manage]{Sithija Manage}\address{Department of Mathematics, Texas A{\&}M University, College Station, TX 77843, U.S.A.}\email{sithijamanage@tamu.edu}
\author[Puente]{Alejandra Puente}\address{Department of Mathematics, Texas A{\&}M University, College Station, TX 77843, U.S.A.}\email{alejandra.puente@tamu.edu}



\begin{abstract}
We discuss the problem of unique determination of the finite free discrete Schr\"{o}dinger operator from its spectrum, also known as Ambarzumian problem, with various boundary conditions, namely any real constant boundary condition at zero and Floquet boundary conditions of any angle. Then we prove the following Ambarzumian-type mixed inverse spectral problem: Diagonal entries except the first and second ones and a set of two consecutive eigenvalues uniquely determine the finite free discrete Schr\"{o}dinger operator.
\end{abstract}
\maketitle

\section{\bf {Introduction}}

The Jacobi matrix is a three-diagonal matrix defined as
\begin{equation*}
 \begin{pmatrix}
b_1 & a_1 & 0 & 0 & 0 \\
a_1 & b_2 & a_2 & \ddots & 0 \\
0  &  a_2 & b_3 & \ddots & 0 \\
0 & \ddots & \ddots & \ddots & a_{n-1} \\
0 & \dots & 0 & a_{n-1} & b_n 
\end{pmatrix}
\end{equation*}
where $n \in \mathbb{N}$, $a_k ~\textgreater~ 0$ for any $k \in \{1,2,\dots,n-1\}$ and $b_k \in \mathbb{R}$ for any $k \in \{1,2,\dots,n\}$. When $a_k=1$ for each $k \in \{1,2,\dots,n-1\}$, this matrix defines the finite discrete Schr\"{o}dinger operator.

Direct spectral problems aim to get spectral information from the sequences $\{a_k\}_{k=1}^{n-1}$ and $\{b_k\}_{k=1}^{n}$. In inverse spectral problems one tries to recover these sequences from spectral information such as the spectrum, the spectral measure or Weyl $m$-function.


Early inverse spectral problems for finite Jacobi matrices appear as discrete analogs of inverse spectral problems for the Schr\"{o}dinger (Sturm-Liouville) equations
\begin{equation*}
 -u''(t) + q(t)u(t) = zu(t),
\end{equation*}
on the interval $[0,\pi]$ with the boundary conditions 
\begin{align*}
  &u(0)\cos\alpha - u'(0)\sin\alpha = 0\\
 &u(\pi)\cos\beta + u'(\pi)\sin\beta = 0,
 \end{align*}
where the potential function $q\in L^1(0,\pi)$ is real-valued and $\alpha, \beta \in [0,\pi)$.

The first inverse spectral result on Schr\"{o}dinger operators is given by Ambarzumian \cite{AMB}. He considered continuous potential with Neumann boundary conditions at both endpoints ($\alpha = \beta = \pi/2$) and showed that $q \equiv 0$ if the spectrum consists of squares of integers. Later Borg \cite{BOR} realized that knowledge of one spectrum is sufficient for unique recovery only for the zero potential. He proved that an $L^1$-potential is uniquely recovered from two spectra, which share the same boundary condition at $\pi$ ($\beta_1=\beta_2$) and one of which is with Dirichlet boundary condition at $0$ ($\alpha_1=0,$ $\alpha_2 \in (0,\pi)$). A few years later, Levinson \cite{LEVI} removed the Dirichlet boundary condition restriction from Borg's result. This famous theorem is also known as two-spectra theorem. Then Marchenko \cite{MAR} observed that the spectral measure (or Weyl-Titchmarsh $m$-function) 
uniquely recovers an $L^1$-potential. Another classical result is due to Hochstadt and Liebermann \cite{HL}, which says that if the first half 
of an $L^1$-potential is known, one spectrum recovers the whole. One can find the statements of these classical theorems and some other 
results from the inverse spectral theory of Schr\"{o}dinger operators e.g. in \cite{HAT} and references therein. 

Finite Jacobi matrix analogs of Borg's and Hochstadt and Lieberman's theorems were considered by Hochstadt \cite{HOC,HOC2,HOC3}, where the potential 
$q$ is replaced by the sequences $\{a_k\}_{k=1}^{n-1}$ and $\{b_k\}_{k=1}^{n}$. These classical theorems led to various other inverse spectral results on finite Jacobi matrices (see \cite{BD,DK,GES,GS,S,WW} and references therein) and other settings such as semi-infinite, infinite, generalized Jacobi matrices and matrix-valued Jacobi operators (see e.g. \cite{CGR,D,D2,DKS2,DKS3,DKS4,DS,GKM,GKT,HAT2,SW,SW2,TES2} and references therein). In general, these problems can be divided into two groups. In Borg-Marchenko-type spectral problems, one tries to recover the sequences $\{a_k\}_{k=1}^{n-1}$ and $\{b_k\}_{k=1}^{n}$ from the spectral data. On the other hand, Hochstadt-Lieberman-type (or mixed) spectral problems recover the sequences $\{a_k\}_{k=1}^{n-1}$ and $\{b_k\}_{k=1}^{n}$ using a mixture of partial information on these sequences and the spectral data. 

Ambarzumian-type problems focus on inverse spectral problems for free discrete Schr\"{o}dinger operators, i.e. $a_k = 1$ and $b_k = 0$ for every $k$, or similar cases when $b_k = 0$ for some $k$. In this paper, we first revisit the classical Ambarzumian problem for the finite discrete Sch\"{o}dinger operator in Theorem \ref{Amb1}, which says that the spectrum of the free operator uniquely determines the operator. Then we provide a counter-example, Example \ref{exmp}, which shows that knowledge of the spectrum of the free operator with a non-zero boundary condition is not sufficient for unique recovery. In Theorem \ref{nzbc}, we observe that a non-zero boundary condition along with the corresponding spectrum of the free operator is needed for the uniqueness result. However, in Theorem \ref{Amb2} we prove that for the free operator with Floquet boundary conditions, the set of eigenvalues including multiplicities is sufficient to get uniqueness up to transpose.   

We also answer the following mixed Ambarzumian-type inverse problem positively in Theorem \ref{Amb3}.
\begin{problem}
Let us define the discrete Schr\"{o}dinger matrix \textbf{S}$_n$ as $a_k = 1$ for $k \in \{1,\dots,n-1\}$ and $b_1,b_2 \in \mathbb{R}$, $b_k = 0$ for $k \in \{3,\dots,n\}$. Let us also denote the free discrete Schr\"{o}dinger operator by \textbf{F}$_n$, which is defined as $a_k = 1$ for $k \in \{1,\dots,n-1\}$ and $b_k = 0$ for $k \in \{1,\dots,n\}$. If \textbf{S}$_n$ and \textbf{F}$_n$ share two consecutive eigenvalues, then do we get $b_1 = b_2 = 0$, i.e. \textbf{S}$_n = \textnormal{\textbf{F}}_n$?
\end{problem}
The paper is organized as follows. In Section 2 we recall necessary definitions and results we use in our proofs. In Section 3 we consider the problem of unique determination of the finite free discrete Schr\"{o}dinger operator from its spectrum, with various boundary conditions, namely any real constant boundary condition at zero, and Floquet boundary conditions of any angle. In Section 4 we prove the above mentioned Ambarzumian-type mixed inverse spectral problem. 

\section{\bf Preliminaries}
Let us start by fixing our notation.
Let \textbf{J}$_n$ represent the finite Jacobi matrix of size $n \times n$
\begin{equation}\label{Jacmat1}
 \textbf{J}_n := \begin{pmatrix}
    b_{1} & a_{1} & {0} & \cdots  & {0} \\
    a_{1} & b_{2} & a_{2} & \ddots  & \vdots \\
    {0} & a_{2} & b_{3} & \ddots  & {0} \\
    \vdots & \ddots & \ddots & \ddots & a_{n-1} \\
    {0} & \cdots & {0} & a_{n-1}  & b_{n} \\ 
\end{pmatrix},\\   
\end{equation}
where $a_k>0,$ $b_k\in \mathbb{R}$. Given \textbf{J}$_n$, let us consider the Jacobi matrix where all $a_k$'s and $b_k$'s are the same as \textbf{J}$_n$ except $b_1$ and $b_n$ are replaced by $b_1+b$ and $b_n+B$ respectively for $b,B \in \mathbb{R}$, i.e.
\begin{equation}\label{JacmatBC}
\textbf{J}_n + b(\delta_1,\cdot)\delta_1+B(\delta_n,\cdot)\delta_n.
\end{equation} 
The Jacobi matrix (\ref{JacmatBC}) is given by the Jacobi difference expression
\begin{equation*}
 a_{k-1}f_{k-1} + b_kf_k + a_kf_{k+1}, \quad \quad k \in \{1,\cdots,n\}
\end{equation*}
with the boundary conditions 
$$
f_0 = bf_1  \qquad \text{and} \qquad f_{n+1} = Bf_n.
$$ 
Let us note that we assume $a_0 = 1$ and $a_n = 1$.

In order to get a unique Jacobi difference expression with boundary conditions for a given Jacobi matrix, we can see the first and the last diagonal entries of the matrix \textbf{J}$_n$, defined in (\ref{Jacmat1}), as the boundary conditions at $0$ and $n+1$ respectively. Therefore let \textbf{J}$_n(b,B)$ denote the Jacobi matrix \textbf{J}$_n$ satisfying $b_1=b$ and $b_n=B$. 

If we consider the Jacobi difference expression with the Floquet boundary conditions 
$$
f_0 = f_ne^{2\pi i\theta} \qquad \text{and} \qquad f_{n+1}=f_1e^{-2\pi i\theta}, \qquad \theta \in [0,\pi)
$$
then we get the following matrix representation, which we denote by \textbf{J}$_n(\theta)$.
\begin{equation}\label{Jacmat}
 \textbf{J}_n(\theta) := \begin{pmatrix}
    b_{1} & a_{1} & {0} & \cdots  & {e^{2\pi i\theta}} \\
    a_{1} & b_{2} & a_{2} & \ddots  & {0} \\
    {0} & a_{2} & b_{3} & \ddots  & {0} \\
    \vdots & \ddots & \ddots & \ddots & a_{n-1} \\
    {e^{-2\pi i\theta}} & {0} & \cdots & a_{n-1}  & b_{n} \\ 
\end{pmatrix}.\\   
\end{equation}

Let us denote the discrete Sch\"{o}dinger matrix accordingly, i.e. \textbf{S}$_n(b,B)$ denotes the matrix \textbf{J}$_n$ such that $b_1=b$, $b_n=B$ and $a_k = 1$ for each $k \in \{1,2,\cdots,n-1\}$. Similarly \textbf{S}$_n(\theta)$ denotes the matrix \textbf{J}$_n(\theta)$ such that $a_k = 1$ for each $k \in \{1,2,\cdots,n-1\}$. 
Let us denote the free discrete Sch\"{o}dinger matrix of size $n \times n$ by $\textbf{F}_n$:
\begin{equation*}
    \textbf{F}_n = \begin{pmatrix}
        0 & 1  & 0 & \cdots & 0\\
        1 & 0  & 1 & \ddots & \vdots\\
        0  & 1 & 0 & \ddots & 0\\
        \vdots & \ddots & \ddots & \ddots & 1\\
        0 & \cdots & 0 & 1 & 0\\
      \end{pmatrix},
\end{equation*}
 so \textbf{F}$_n(b,B)$ and \textbf{F}$_n(\theta)$ denote the free discrete Sch\"{o}dinger matrices with boundary conditions $b$ at $0$, $B$ at $n+1$ and Floquet boundary conditions for $\theta$, respectively.
 
Let us state some basic properties of the free discrete Sch\"{o}dinger matrix. If $\lambda_1, \lambda_2, ..., \lambda_n$ denote the eigenvalues of $\textbf{F}_n$, they have the following properties:
\begin{itemize}
    \item For all $k$, $\lambda_k \in [-2, 2]$.
    \item The free discrete Sch\"{o}dinger matrix \textbf{F}$_n$ has $n$ distinct eigenvalues, so we can reorder the eigenvalues such that $\lambda_1 < \lambda_2 < \cdots <
    \lambda_n$.
    \item Let \textbf{F}$_{n-1}$ be the $(n-1) \times (n-1)$ submatrix of \textbf{F}$_n$ obtained by removing the last row and the last column of \textbf{F}$_{n}$. If $\mu_1, \mu_2, \cdots, \mu_{n-1}$ denote the eigenvalues of \textbf{F}$_{n-1}$ taken in increasing order, then we have the interlacing property of eigenvalues, i.e.
    \begin{equation*}
        \lambda_1 < \mu_1 < \lambda_2 < \mu_2 < ... < \lambda_{n-1} < \mu_{n-1} < \lambda_n
    \end{equation*}
\end{itemize}

The second and third properties are valid for any Jacobi matrix \textbf{J}$_{n}$. These basic properties can be found in \cite{TES}, which provides an extensive study of Jacobi operators.

The following results show smoothness of simple eigenvalues and corresponding eigenvectors of a smooth matrix-valued function. We will use them in Section 4 in order to prove the mixed inverse spectral problem mentioned in Introduction.

\begin{theorem}\emph{(}\cite{LAX}, Theorem 9.7\emph{)}\label{smth1}
Let $A(t)$ be a differentiable square matrix-valued function of the real variable $t$. Suppose that $A(0)$ has an eigenvalue $a_0$ of multiplicity one, in the sense that $a_0$ is a simple root of the characteristic polynomial of $A(0)$. Then for $t$ small enough, $A(t)$ has an eigenvalue $a(t)$ that depends differentiably on t, and which equals $a_o$ at zero, that is, $a(0) = a_0$.
\end{theorem}

\begin{theorem}\emph{(}\cite{LAX}, Theorem 9.8\emph{)}\label{smth2}
Let $A(t)$ be a differentiable matrix-valued function of $t$, $a(t)$ an eigenvalue of $A(t)$ of multiplicity one. Then we can choose an eigenvector $X(t)$ of $A(t)$ pertaining to the eigenvalue $a(t)$ to depend differentiably on $t$.
\end{theorem}

Once we obtain smoothness of an eigenvalue and the corresponding eigenvector of a smooth self-adjoint matrix, the Hellmann-Feynman theorem relates the derivatives of the eigenvalue and the matrix with the corresponding eigenvector. 

\begin{theorem}[Hellmann-Feynman]\emph{(}\cite{SIM2}, Theorem 1.4.7\emph{)}\label{hf}
Let $A(t)$ be self-adjoint matrix-valued, $X(t)$ be vector-valued and $\lambda(t)$ be real-valued functions. If $A(t)X(t) = \lambda(t)X(t)$ and $||X(t)|| = 1$, then 
$$
\lambda'(t) = \langle X(t) , A'(t)X(t) \rangle.
$$

\end{theorem}


 \section{\bf {Ambarzumian problem with various boundary conditions}}
 
 In addition to the notations we introduced in the previous section, let $p_n(x)$ be the characteristic polynomial of $\textbf{F}_n$ with zeroes $\lambda_1,\dots,\lambda_n$ and let $q_n(x)$ be the characteristic polynomial of $\textbf{S}_n$ with zeroes $\mu_1,\dots,\mu_n$.

Let us start by obtaining the first three leading coefficients of $q_n(x)$. This is a well-known result, but we give a proof in order make this section self-contained.
\begin{lemma}\label{charpol}
The characteristic polynomial $q_n(x)$ of the discrete Schr\"{o}dinger matrix \textbf{\emph{S}}$_n$ has the form
$$q_n(x)=x^n-\left(\sum_{i=1}^n b_i\right)x^{n-1}+\left(\sum_{\substack{1\le i<j\le n}}b_ib_j-(n-1)\right)x^{n-2} + Q_{n-2}(x),$$
where $Q_{n-2}(x)$ is a polynomial of degree at most $n-2$.
\end{lemma}

\begin{proof}
The characteristic polynomial $q_n(x)$ is given by $\det(x\textbf{I}_n - \textbf{S}_n)$. Let us consider Liebniz' formula for the determinants
\begin{equation}\label{Lebfor}
    \det(\textbf{A})=\sum_{\sigma\in \mathbb{S}_n}\textrm{sgn}(\sigma)\prod_{i=1}^n\alpha_{i,\sigma(i)},
\end{equation}
where $\textbf{A} = [\alpha_{i,j}]$ is an $n \times n$ matrix and \textrm{sgn} is the sign function of permutations in the permutation group $\mathbb{S}_n$, which returns $+1$ and $-1$ for even and odd permutations, respectively. If we use (\ref{Lebfor}) with the identity permutation, then $\det(x\textbf{I}_n - \textbf{S}_n)$ becomes $\prod_{i=1}^n (x-b_i)$, so we can get that $q_n(x)$ is a monic polynomial and the coefficient of $x^{n-1}$ is $-\operatorname{tr}(\textbf{S}_n)$. 

The coefficient of the $x^{n-2}$ term of $q_n(x)$ is formed from the sum of all disjoint pairs of $a_i$. Hence we obtain $\sum_{1 \leq i < j\leq n}b_ib_j$.

The only other permutation that will yield an $x^{n-2}$ term is a transposition. However, if $|i-j|\geq 1$, then the product will be zero. Thus we are looking for transpositions where $|i-j|=1$. There are $n-1$ of these, namely $(1,2),(2,3),\dots,(n-2,n-1),(n-1,n)$. The product is of the form
\begin{equation*}
        \prod_{i=1}^n \alpha_{i, \sigma(i)}  = \frac{(-1)(-1)(x-b_1)(x-b_2)\cdots(x-b_n)}{(x-b_i)(x-b_j)} 
         = x^{n-2} + r_{n-3}(x),
\end{equation*}
where $r_{n-3}(x)$ is a polynomial of degree at most $n-3$. Since the signature of a transposition is negative, we derive $-x^{n-2}$ for each product. Summing over all $n-1$ permutations and adding to $\sum_{1\leq i < j\leq n}b_ib_j$ yields our desired result.
\end{proof}

\begin{corollary}\label{cor1}
The characteristic polynomial $p_n(x)$ of the free discrete Schr\"{o}dinger matrix \textbf{\emph{F}}$_n$ has the form
$$p_n(x)=x^n-(n-1)x^{n-2} + P_{n-2}(x),$$
where $P_{n-2}(x)$ is a polynomial of degree at most $n-2$.
\end{corollary} 
\begin{proof}
Simply set $b_i = 0$ for each $i \in \{1,\cdots,n\}$ and apply Lemma \ref{charpol}.
\end{proof}

Let us start by giving a proof of Ambarzumian problem with Dirichlet-Dirichlet boundary conditions, i.e. for the matrix \textbf{F}$_n(0,0)$ = \textbf{F}$_n$ in our notation.

\begin{theorem}\label{Amb1}
 Suppose \textbf{\emph{S}}$_n$ shares all of its eigenvalues with \textbf{\emph{F}}$_n$. Then \textbf{\emph{S}}$_n=~$\textbf{\emph{F}}$_n$.
 \end{theorem}
\begin{proof}
In order for the two matrices to have all the same eigenvalues, they must have equal characteristic polynomials. Comparing the results from Lemma \ref{charpol} to Corollary \ref{cor1}, we must have 
\begin{equation}
    \sum_{i=1}^{n}b_i=0\qquad \text{and} \qquad
    \sum_{1\leq i < j\leq n}b_ib_j=0
\end{equation}
This leads us to conclude that 
\begin{equation*}
     \sum_{i=1}^{n} b_i^2 = \left(\sum_{i=1}^{n} b_i\right)^2-2\left(\sum_{1\leq i < j\leq n}b_ib_j\right)=0
\end{equation*}
which only occurs when all the $b_i$ are zero, i.e. $\textbf{S}_n=\textbf{F}_n$, since $b_i \in \mathbb{R}$ for each $i \in \{1,\cdots,n\}$.
\end{proof}

A natural question to ask is whether or not we get the uniqueness of the free operator with non-zero boundary conditions. At this point let us recall Borg and Levinson's famous two-spectra theorem \cite{BOR,LEVI}, which says that two spectra for different boundary conditions of a regular Schr\"{o}dinger operator on a finite interval uniquely determines the operator. Finite Jacobi analogs of two-spectra theorem were proved by Gesztesy and Simon.

Given a Jacobi matrix \textbf{J}$_n$, define \textbf{J}$_n^{(b)}$ as the Jacobi matrix where all $a_k$'s and $b_k$'s are the same as \textbf{J}$_n$ except $b_1$ is replaced by $b_1+b$ for $b \in \mathbb{R}$, i.e.
\begin{equation}
\textbf{J}_n^{(b)} := \textbf{J}_n + b(\delta_1,\cdot)\delta_1.
\end{equation}

Let us denote the eigenvalues of \textbf{J}$_n^{(b)}$ by $\lambda^{(b)}_k$ for $k \in \{1,2,\cdots,n\}$. Note that \textbf{J}$_n$ and \textbf{J}$_n^{(b)}$ represent the same Jacobi difference equation with different boundary conditions, namely \textbf{J}$_n$ with boundary conditions $b_1$ at $0$, $b_n$ at $n+1$ and \textbf{J}$_n^{(b)}$ with boundary conditions $b_1+b$ at $0$, $b_n$ at $n+1$. The Jacobi matrix \textbf{J}$_n^{(b)}$ can also be seen as a rank-one perturbation of the Jacobi matrix \textbf{J}$_n$.

Gesztesy and Simon \cite{GS} proved that if $b$ is known, then the spectrum of \textbf{J}$_n$ and the spectrum of \textbf{J}$_n^{(b)}$ except one eigenvalue uniquely determine the Jacobi matrix. 

\begin{theorem}\emph{(}\cite{GS}, Theorem 5.1\emph{)}\label{GesSim1}
The eigenvalues $\lambda_1,\cdots,\lambda_n$ of \textbf{\emph{J}}$_n$, together with $b$ and $n-1$ eigenvalues $\lambda^{(b)}_1,\cdots,\lambda^{(b)}_{n-1}$ of \textbf{\emph{J}}$_n^{(b)}$, determine \textbf{\emph{J}}$_n$ uniquely.
\end{theorem}

Note that it is irrelevant which $n-1$ eigenvalues from the spectrum of \textbf{J}$_n^{(b)}$ are known. Gesztesy and Simon also showed that if two spectra are known, the uniqueness result is obtained without knowing $b$.

\begin{theorem}\emph{(}\cite{GS}, Theorem 5.2\emph{)}\label{GesSim2}
The eigenvalues $\lambda_1,\cdots,\lambda_n$ of \textbf{\emph{J}}$_n$, together with the $n$ eigenvalues $\lambda^{(b)}_1,\cdots,\lambda^{(b)}_{n}$ of some \textbf{\emph{J}}$_n^{(b)}$ (with $b$ unknown), determine \textbf{\emph{J}}$_n$ and $b$ uniquely.
\end{theorem}

After Theorems \ref{Amb1}, \ref{GesSim1} and \ref{GesSim2} one may expect to get the uniqueness of a free discrete Schr\"{o}dinger operator from a spectrum with non-zero boundary condition at $0$. However, this is not the case because of the following counterexample:
\begin{example}\label{exmp}
Let us define the discrete Schr\"{o}dinger matrices
\center
$A :=
\begin{pmatrix}
2 & 1 & 0\\
1 & 0 & 1\\
0 & 1 & 0
\end{pmatrix}
\text{ and } 
B :=
\begin{pmatrix}
-2/(1+\sqrt{5}) & 1 & 0\\
1 & 1 & 1\\
0 & 1 & (1+\sqrt{5})/2\\
\end{pmatrix}$.
\flushleft 

The matrices $A$ and $B$ have the same characteristic polynomial $x^3-2x^2-2x+2$, so they share the same spectrum.
\end{example}

This example shows that Theorem \ref{Amb1} was a special case, so in order to get uniqueness of a rank-one perturbation of the free operator, we also need to know the non-zero boundary condition along with the spectrum.

\begin{theorem}\label{nzbc}
 Suppose \textbf{\emph{S}}$_n(b,b_n)$ shares all of its eigenvalues with \textbf{\emph{F}}$_n(b,0)$. Then \textbf{\emph{S}}$_n(b,b_n)=~$\textbf{\emph{F}}$_n(b,0)$.
\end{theorem}
\begin{proof}
Comparing coefficients of characteristic polynomials of \textbf{S}$_n(b,b_n)$ and \textbf{F}$_n(b,0)$ like we did in the proof of Theorem \ref{Amb1}, we get
\begin{equation}\label{coefEq}
    b+\sum_{i=2}^{n}b_i=b\qquad \text{and} \qquad
    b\sum_{j=2}^n b_j + \sum_{2\leq i < j\leq n}b_ib_j=0
\end{equation}
The first equation of (\ref{coefEq}) gives $\sum_{i=2}^{n}b_i=0$ and using this in the second equation of (\ref{coefEq}) we get $\sum_{2\leq i < j\leq n}b_ib_j = 0$. This leads us to conclude that 
\begin{equation*}
     \sum_{i=2}^{n} b_i^2 = \left(\sum_{i=2}^{n} b_i\right)^2-2\left(\sum_{2\leq i < j\leq n}b_ib_j\right)=0
\end{equation*}
which only occurs when $b_i = 0$ for each $i \in \{2,\cdots,n\}$, i.e. \textbf{S}$_n(b,b_n)$ = \textbf{F}$_n(b,0)$.
\end{proof}

Now, we approach the Ambarzumian problem with Floquet boundary conditions. Let us recall that \textbf{S}$_n(\theta)$ and \textbf{F}$_n(\phi)$ denote a discrete Schr\"odinger operator and the free discrete Schr\"odinger Operator with Floquet Boundary Conditions for the angles $0 \leq \theta < 1$ and $0 \leq \phi < 1$, respectively:

$$
\textbf{S}_n(\theta) = \begin{pmatrix}
    b_{1} & 1 & {0} & \cdots  & {e^{2\pi i\theta}} \\
    1 & b_{2} & 1 & \ddots  & {0} \\
    {0} & 1 & b_{3} & \ddots  & {0} \\
    \vdots & \ddots & \ddots & \ddots & 1 \\
    {e^{-2\pi i\theta}} & {0} & \cdots & 1  & b_{n} \\ 
\end{pmatrix},\quad
\textbf{F}_n(\phi) = \begin{pmatrix}
    0 & 1 & {0} & \cdots  & {e^{2\pi i\phi}} \\
    1 & 0 & 1 & \ddots  & {0} \\
    {0} & 1 & 0 & \ddots  & {0} \\
    \vdots & \ddots & \ddots & \ddots & 1 \\
    {e^{-2\pi i\phi}} & {0} & \cdots & 1  & 0 \\ 
\end{pmatrix}.\\[0.1in]
$$
The following theorem shows that with Floquet boundary conditions, the knowledge of the spectrum of the free operator is sufficient for the uniqueness of the operator up to transpose.

\begin{theorem}\label{Amb2}
    Suppose that \textbf{\emph{S}}$_n(\theta)$ shares all of its eigenvalues with \textbf{\emph{F}}$_n(\phi)$, including multiplicity, for $0 \leq \theta,\phi < 1$. Then $b_1=\dots=b_n=0$ and $\theta = \phi$ or $\theta=1-\phi$, i.e. \textbf{\emph{S}}$_n(\theta)$ = \textbf{\emph{F}}$_n(\phi)$ or \textbf{\emph{S}}$_n(\theta)$ = \textbf{\emph{F}}$_n^{ \mathsf{T}}(\phi)$
\end{theorem}
\begin{proof}
     Let us define $D[k,l]$ as the following determinant of a $(l-k+1) \times (l-k+1)$ matrix for $1\leq k < l \leq n$:
    
    \begin{equation*}
        D[k,l]:=\begin{vmatrix}
            x-b_k & -1  & 0 & \cdots & 0\\
        -1 & x-b_{k+1} & \ddots & \ddots & \vdots\\
        0 & \ddots & \ddots & \ddots & 0\\
        \vdots & \ddots &  \ddots & x-b_{l-1} & -1\\
        0 & \cdots & 0 & -1 & x-b_l\\
        \end{vmatrix}_{(l-k+1)}\\[0.1in]
    \end{equation*}
    Let us consider the characteristic polynomial of \textbf{S}$_n(\theta)$ by using cofactor expansion on the first row:
     \begin{equation*}
     \begin{aligned}
        |x\textbf{I}_n-\textbf{S}_n(\theta)|&=(x-b_1)D[2,n]+\begin{vmatrix}
        -1 & -1  & 0 & \cdots & 0\\
        -1 & x-b_{3} & -1 & \ddots & \vdots\\
        0 & -1 & \ddots & \ddots & 0\\
        \vdots & \ddots &  \ddots & x-b_{n-1} & -1\\
        -e^{-2\pi i \theta} & 0 & \cdots & -1 & x-b_n\\
        \end{vmatrix}_{(n-1)}
        \\
        &\\
        &+(-1)^{n+1}\left(-e^{2\pi i \theta}\right)
        \begin{vmatrix}
            -1 & x-b_{2} & -1 & 0 & \cdots \\
        0 & -1 & x-b_{3} & \ddots & \ddots \\
        \vdots  & \ddots & \ddots & \ddots & -1\\
        0 & \cdots & 0 & -1 & x-b_{n-1}\\
        -e^{-2\pi i \theta} & 0 & \cdots & 0 & -1\\
        \end{vmatrix}_{(n-1)}
    \end{aligned}\\[0.1in]
    \end{equation*}
    Then by using cofactor expansions on the first row of the determinant in the second term and on the first column of the determinant in the third term we get
    
    \begin{equation*}
     \begin{aligned}
        |x\textbf{I}_n-\textbf{S}_n(\theta)|&=(x-b_1)D[2,n]+(-1)D[3,n]+\begin{vmatrix}
            0 & -1 & 0 & \cdots & 0\\
        0 & x-b_{4} & -1 & \ddots & \vdots \\
        0 & -1 & \ddots & \ddots & 0\\
        \vdots & \ddots & \ddots & \ddots & -1\\
        -e^{-2\pi i \theta} & 0 & \cdots & -1 & x-b_{n}\\
        \end{vmatrix}_{(n-2)}
        \\
        &\\
        &+(-1)^{n+1}\left(-e^{2\pi i \theta}\right)
        (-1)
        \begin{vmatrix}
            -1 & x-b_{3} & -1 & 0 & \cdots \\
        0 & -1 & x-b_{4} & \ddots & \ddots \\
        \vdots  & \ddots & \ddots & \ddots & -1\\
        0 & \cdots & \ddots & -1 & x-b_{n-1}\\
        0 & \cdots & \cdots & 0 & -1\\
        \end{vmatrix}_{(n-2)}
        &\\
        &+(-1)^{n+1}\left(-e^{2\pi i \theta}\right)(-1)^n\left(-e^{-2\pi i \theta}\right)D[2,n-1]
    \end{aligned}\\[0.1in]
    \end{equation*}
    Now let's use cofactor expansion on the first column of the determinant in the third term. Also note that the determinant in the fourth term is the determinant of an upper triangular matrix. Therefore,
    
    \begin{equation*}
     \begin{aligned}
        |x\textbf{I}_n-\textbf{S}_n(\theta)|&=(x-b_1)D[2,n]-D[3,n]\\
        &\\
        &+(-1)^{n-1}\left(-e^{2\pi i \theta}\right)\begin{vmatrix}
        -1 & 0 & 0 & \cdots & 0 \\
        x-b_4 & -1 & \ddots & \ddots & \vdots \\
        -1  & x-b_5 & \ddots & \ddots & 0\\
        0 & \ddots & \ddots & -1 & 0\\
        -e^{-2\pi i \theta} & 0 & -1 & x-b_{n-1} & -1\\
        \end{vmatrix}_{(n-3)}
        \\
        &\\
        &+(-1)^{n+1}\left(-e^{2\pi i \theta}\right)\left[(-1)(-1)^{n-2}+(-1)^n\left(-e^{-2\pi i \theta}\right)D[2,n-1]\right]
    \end{aligned}\\[0.1in]
    \end{equation*}
    Finally, noting again that the determinant in the third term is that of a lower triangular matrix, we get
    \begin{equation}\label{detFloquet}
     \begin{aligned}
        |x\textbf{I}_n-\textbf{S}_n(\theta)|&=(x-b_1)D[2,n]-D[3,n]
        +(-1)^{n-1}\left(-e^{2\pi i \theta}\right)(-1)^{n-3}\\
        &+(-1)^{2n}\left(-e^{-2\pi i \theta}\right)+(-1)^{2n+1}D[2,n-1]\\
        &=(x-b_1)D[2,n]-D[3,n]-D[2,n-1]
        -e^{2\pi i \theta}-e^{-2\pi i \theta}
    \end{aligned}
    \end{equation}
    At this point note that $D[k,l]$ is the characteristic polynomial of the following discrete Schr\"{o}dinger matrix
    \begin{equation*}
 \begin{pmatrix}
    b_{k} & 1 & {0} & \cdots  & {0} \\
    1 & b_{k+1} & 1 & \ddots  & \vdots \\
    {0} & 1 & \ddots & \ddots  & {0} \\
    \vdots & \ddots & \ddots & b_{l-1} &  1 \\
    {0} & \cdots & {0} & 1  & b_{l} \\ 
\end{pmatrix}.\\[0.1in]   
\end{equation*}
Therefore using Lemma \ref{charpol} and equation (\ref{detFloquet}), we obtain
    \begin{equation}\label{Acharpoly}
    \begin{aligned}
        |x\textbf{I}_n-\textbf{S}_n(\theta)|
        &=x^n-\left(\sum_{i=1}^n b_i\right)x^{n-1}+\left(\sum_{\substack{1\leq i<j\leq n}}b_ib_j-(n-1)\right)x^{n-2}\\
        &+f_{n-3}(x)-e^{2\pi i \theta}-e^{-2\pi i \theta}
    \end{aligned}
    \end{equation}
    where $f_{n-3}$ is a polynomial of degree at most $n-3$, which is independent of $\theta$. Using the same steps for $\textbf{F}_n(\phi)$, we obtain
    \begin{equation}\label{Scharpoly}
    \begin{aligned}
        |x\textbf{I}_n-\textbf{F}_n(\phi)|
        &=x^n-(n-1)x^{n-2}+g_{n-3}(x)-e^{2\pi i \phi}-e^{-2\pi i \phi}
    \end{aligned}
    \end{equation}
    where $g_{n-3}$ is a polynomial of degree at most $n-3$, which is independent of $\phi$.
    
     Comparing equations (\ref{Acharpoly}) and (\ref{Scharpoly}), like we did in the proof of Theorem \ref{Amb1}, we can conclude that the diagonal entries $\{b_i\}_{i=1}^n$ of \textbf{S}$_n(\theta)$ must be zero.
    
    Note that the expression consisting of the first three terms in the right end of (\ref{detFloquet}), $(x-b_1)D[2,n]-D[3,n]-D[2,n-1]$ is independent of $\theta$. In addition, we observed that $b_1=\dots=b_n=0$. Therefore using the equivalence of the characteristic polynomials of \textbf{S}$_n(\theta)$ and \textbf{F}$_n(\phi)$, we obtain
    $$
    e^{2\pi i \theta}+e^{-2\pi i \theta} = e^{2\pi i \phi}+e^{-2\pi i \phi},
    $$
    which can be written using Euler's identity as
    \begin{equation}\label{eulerid}
    2\cos(2\pi\theta) = 2\cos(2\pi\phi).
    \end{equation}
    Equation (\ref{eulerid}) is valid if and only if $\theta$ differs from $\phi$ or $-\phi$ by an integer. Since $0 \leq \theta,\phi < 1$, the only possible values for $\theta$ are $\phi$ and $1-\phi$. This completes the proof.
\end{proof}
 
\section{\bf {An Ambarzumian-type mixed inverse spectral problem}}

    Let us introduce the following $n \times n$ discrete Schr\"{o}dinger matrix for $1 \leq m \leq n$:
    
$$
\textnormal{\textbf{S}}_{n,m} := \begin{pmatrix}
    b_{1}  &  1   &  0   & 0 &  \ldots  &  0 \\
    1      &\ddots &  1   & 0 & \ddots  &  0 \\
    0      &  1   &  b_{m}   &  1 & \ddots  & \vdots \\
    0 & 0 & 1 &  0  &  \ddots & 0 \\
    \vdots &\ddots&\ddots&  \ddots  & \ddots &  1\\
    0      &\ldots&  \ldots & 0   &    1     &  0\\
    \end{pmatrix}\\[0.1in]
$$
    
     Let us also recall that \textbf{F}$_{n}$ denotes the free discrete Schr\"{o}dinger matrix of size $n \times n$. In this section our goal is to answer the following Ambarzumian-type mixed spectral problem positively for the $m=2$ case.
    
    \begin{problem}
 If \textbf{S}$_{n,m}$ and \textbf{F}$_{n}$ share $m$ consecutive eigenvalues, then do we get $b_1 = \cdots = b_m = 0$, i.e. \textbf{S}$_{n,m} = \textnormal{\textbf{F}}_n$?
\end{problem}

When $m=1$, this problem becomes a special case of the following result of Gesztesy and Simon \cite{GS}. For a Jacobi matrix given as (\ref{Jacmat1}), let us consider the sequences $\{a_k\}$ and $\{b_k\}$ as a single sequence $b_1,a_1,b_2,a_2,\cdots,a_{n-1},b_n = c_1,c_2,\cdots,c_{2n-1}$, i.e. $c_{2k-1} := b_k$ and $c_{2k} := a_k$. 

\begin{theorem}\emph{(}\cite{GS}, Theorem 4.2\emph{)}
Suppose that $1 \leq k \leq n$ and $c_{k+1}, \cdots ,c_{2n-1}$ are known, as
well as $k$ of the eigenvalues. Then $c_1,\cdots,c_k$ are uniquely determined.
\end{theorem}
    
    By letting $k=1$, we get the inverse spectral problem stated above for $m=1$. Now let us prove the $m=2$ case. Let $\lambda_{1}<\lambda_{2}<\dots<\lambda_{n}$ denote the eigenvalues of \textbf{F}$_{n}$, and let $\tilde{\lambda}_1<\tilde{\lambda}_2<\dots<\tilde{\lambda}_n$ denote the
    eigenvalues of \textbf{S}$_{n,2}$.  
    
    \begin{theorem}\label{Amb3}
    Let $\lambda_{k} = \tilde{\lambda}_{k}$ and $\lambda_{k+1} = \tilde{\lambda}_{k+1}$ for some $k \in \{1,2,\dots,n-1\}$. Then $b_1 = 0$ and $b_2 = 0$, i.e. \textnormal{\textbf{S}}$_{n,2} = \textnormal{\textbf{F}}_n$.
    \end{theorem}
    
    \begin{proof}
    For simplicity, let us use \textbf{S}$_{n}$ instead of \textbf{S}$_{n,2}$.
    We start by proving the following claim.\\
    
    \underline{\textbf{Claim:}} If $\lambda_k = \tilde{\lambda}_{k}$ and $\lambda_{k+1} = \tilde{\lambda}_{k+1}$ , then either $b_{1} = b_{2} = 0$ or $b_{1} = \lambda_{k}+\lambda_{k+1}$ and $b_{2} = 1/\lambda_{k}+1/\lambda_{k+1}$.\\

Let us consider the characteristic polynomial of \textbf{S}$_{n}$ using cofactor expansion on the last row of $\lambda I-\textnormal{\textbf{S}}_{n}$.
\begin{align*}
    &\det(\lambda I-\textnormal{\textbf{S}}_{n}) = \\
    &(\lambda - b_{1}) 
    \begin{vmatrix}
    \lambda -b_{2}&  -1  &  0   &  \ldots  &  0\\
    -1     &\lambda& -1   &  \ddots  &  \vdots\\
    0      & -1   &\lambda&  \ddots  &  0\\
    \vdots &\ddots&\ddots&  \ddots  &  -1\\
    0      &\ldots&  0   &   -1     &  \lambda\\
    \end{vmatrix}_{(n-1)}+\quad
    \begin{vmatrix}
    -1     &  -1  &  0   &  \ldots  &  0\\
    -1     &\lambda& -1   &  \ddots  &  \vdots\\
    0      & -1   &\lambda&  \ddots  &  0\\
    \vdots &\ddots&\ddots&  \ddots  &  -1\\
    0      &\ldots&  0   &   -1     &  \lambda\\
    \end{vmatrix}_{(n-1)}
\end{align*}

Using cofactor expansion on the first row for the first term and the first column for the second term, we get
\begin{align*}
    \det(\lambda I-\textnormal{\textbf{S}}_{n}) &=
    (\lambda - b_{1})(\lambda - b_{2})
    \det(\lambda I-\textnormal{\textbf{F}}_{n-2})\\
    &+(\lambda -b_{1})
    \begin{vmatrix}
    -1     &  -1  &  0   &  \ldots  &  0\\
    0    &\lambda& -1   &  \ddots  &  \vdots\\
    0      & -1   &\lambda&  \ddots  &  0\\
    \vdots &\ddots&\ddots&  \ddots  &  -1\\
    0      &\ldots&  0   &   -1     &  \lambda\\
    \end{vmatrix}_{(n-2)}
    -
    \det(\lambda I-\textnormal{\textbf{F}}_{n-2})
\end{align*}

Finally, using cofactor expansion on the first column of the second term, we get
\begin{equation}\label{eveqn}
    \det(\lambda I-\textnormal{\textbf{S}}_{n}) = [(\lambda - b_{1})(\lambda - b_{2})-1]
    \det(\lambda I-\textnormal{\textbf{F}}_{n-2})-
    (\lambda -b_{1})\det(\lambda I-\textnormal{\textbf{F}}_{n-3}).
\end{equation}

Since $\tilde{\lambda}_{k}=\lambda_{k}$ and $\tilde{\lambda}_{k+1}=\lambda_{k+1}$, right hand side of (\ref{eveqn}) is zero when $\lambda = \lambda_{k}$ or $\lambda = \lambda_{k+1}$. Therefore for $\lambda = \lambda_{k}$ or $\lambda = \lambda_{k+1}$ we get

\begin{equation}\label{cpeq}
    \frac{(\lambda -b_{1})(\lambda -b_{2})-1}{\lambda - b_{1}} = 
    \frac{\det(\lambda I-\textnormal{\textbf{F}}_{n-3})}{\det(\lambda I-\textnormal{\textbf{F}}_{n-2})}.
\end{equation} 

Note that equation (\ref{cpeq}) is also valid for $\textnormal{\textbf{F}}_{n}$, i.e. when $b_1=b_2=0$, and the right hand side of the equation does not depend on $b_{1}$ or $b_{2}$ and hence identical for $\textnormal{\textbf{S}}_{n}$ and $\textnormal{\textbf{F}}_{n}$. Therefore the left hand side of (\ref{cpeq}) should also be identical for $\textnormal{\textbf{S}}_{n}$ and $\textnormal{\textbf{F}}_{n}$, when $\tilde{\lambda}_{k}=\lambda_{k}$ and $\tilde{\lambda}_{k+1}=\lambda_{k+1}$. Hence, 
\begin{equation}
    \frac{(\lambda -b_{1})(\lambda -b_{2})-1}{\lambda - b_{1}} = \frac{(\lambda -0)(\lambda -0)-1}{\lambda - 0} 
\end{equation} for $\lambda = \lambda_{k}$ or $\lambda = \lambda_{k+1}$. Therefore, 
\begin{align*}
    \lambda (\lambda -b_{1})(\lambda -b_{2}) - \lambda &= 
    (\lambda^{2}-1)(\lambda-b_{1})\\
    \lambda^{3}-(b_{1}+b_{2})\lambda^{2}+b_{1}b_{2}\lambda - \lambda
    &= \lambda^{3} -b_{1}\lambda^{2} - \lambda + b_{1}\\
-b_{2}\lambda^{2}+b_{1}b_{2}\lambda - b_{1} &= 0 
\end{align*}
for $\lambda = \lambda_{1}$ or $\lambda = \lambda_{2}$. If $b_{2} = 0$, then $b_{1} = 0$ from the last equation above, so we can assume $b_{2} \neq 0$. Then $\lambda^{2}-b_{1}\lambda + b_{1}/b_{2} = 0$ for $\lambda = \lambda_{k}$ or $\lambda = \lambda_{k+1}$.

Since $x^{2}-b_{1}x+b_{1}/b_{2}$ is a monic polynomial with two distinct roots $x=\lambda_{k}$ and $x=\lambda_{k+1}$, we get

\begin{equation*}
    x^{2}-b_{1}x+b_{1}/b_{2} = (x-\lambda_{k})(x-\lambda_{k+1}) 
\end{equation*} which implies
\begin{equation*}
    x^{2}-b_{1}x+b_{1}/b_{2} =
    x-(\lambda_{k}+\lambda_{k+1})x + \lambda_{k}\lambda_{k+1}
\end{equation*}
Comparing coefficients we get our claim, since $b_{1} = \lambda_{k}+\lambda_{k+1}$, and $b_{1}/b_{2} = \lambda_{k}\lambda_{k+1}$ implies 
$$b_2 = \frac{b_{1}}{\lambda_{k}\lambda_{k+1}} = \frac{\lambda_{k}+ \lambda_{k+1}}{\lambda_{k}\lambda_{k+1}} = \frac{1}{\lambda_{k}}+\frac{1}{\lambda_{k+1}}
$$

Now our goal is to get a contradiction for the second case of the claim, i.e. when 
$b_{1} = \lambda_{k}+\lambda_{k+1}$ and $b_2 = 1/\lambda_{k}+1/\lambda_{k+1}$, so let us assume 
$$
b_{1} = \lambda_{k}+\lambda_{k+1} \quad \text{and} \quad b_2 = 1/\lambda_{k}+1/\lambda_{k+1}.
$$

First let us show that $b_1$ and $b_2$ have the same sign. If $n$ is even and $k = n/2$, then $\lambda_k = -\lambda_{k+1}$. Hence $b_1 = b_2 = 0$. If $n$ is odd and $k = (n-1)/2$ or $k = (n+1)/2$, then one of the eigenvalues $\lambda_k$ or $\lambda_{k+1}$ is zero, so $b_2$ is undefined. For all other values of $k$, two consecutive eigenvalues $\lambda_k$ and $\lambda_{k+1}$ and hence $b_1$ and $b_2$ have the same sign.

Without loss of generality let us assume both $\lambda_k$ and $\lambda_{k+1}$ are negative and $b_1 \leq b_2$. Let us define the matrices $\textnormal{\textbf{C}}_{n}$ and $\textnormal{\textbf{M}}_{n}(t)$ with the real parameter $t$ as follows:

$$
\textnormal{\textbf{C}}_{n} := \begin{pmatrix}
    b_{2}  &  1 & 0    &  \ldots  &  0\\
    1      &b_{2} & 1  &  \ddots  &  \vdots\\
    0      &  1 & 0    &  \ddots  &  0\\
    \vdots &\ddots &\ddots &  \ddots  &  1\\
    0      &\ldots  & 0 &  1     &  0\\
    \end{pmatrix} \quad \text{and} \quad
\textnormal{\textbf{M}}_{n}(t) := \begin{pmatrix}
    -t      &  1   &  0   &  \ldots  &  0\\
    1      &  -t  &  1   &  \ddots  &  \vdots\\
    0      &  1   &  0   &  \ddots  &  0\\
    \vdots &\ddots&\ddots&  \ddots  &  1\\
    0      &\ldots&  0   &    1     &  0\\
    \end{pmatrix}.\\[0.1in]
$$

Note that kth eigenvalue of $\textnormal{\textbf{C}}_n$, denoted by $\lambda_k(\textnormal{\textbf{C}}_n)$, is greater than or equal to $\tilde{\lambda}_k$, since $\textnormal{\textbf{C}}_n \geq \textnormal{\textbf{A}}_n$. Let us also note that $\textnormal{\textbf{M}}_n(-b_2) = \textnormal{\textbf{C}}_n$ and $\textnormal{\textbf{M}}_n(0) = \textnormal{\textbf{F}}_n$. Let us denote the kth eigenvalue of $\textnormal{\textbf{M}}_n(t)$ by $\lambda_k(t)$ and the corresponding eigenvector by $X(t)$, normalized as $||X(t)||=1$. Since $\textnormal{\textbf{M}}_n(t)$ is a smooth function of $t$ around $0$, same is true for $\lambda_k(t)$ and $X(t)$ by Theorem \ref{smth1} and Theorem \ref{smth2}. Let us recall that $\textnormal{\textbf{M}}_n(t)$ is self-adjoint, $||X(t)|| = 1$ and $\textnormal{\textbf{M}}_n(t)X(t) = \lambda_k(t)X(t)$. Therefore by  Theorem \ref{hf}, the Hellmann-Feynman Theorem, we get 
\begin{equation}\label{derofev}
\lambda_k'(t) = \langle X(t),\textnormal{\textbf{M}}_n'(t)X(t) \rangle = -X_1^2(t) - X_2^2(t),
\end{equation}
where $X(t)^T = [X_1(t),X_2(t),\cdots,X_n(t)]$. Since $X(t)$ is a non-zero eigenvector of the tridiagonal matrix $\textnormal{\textbf{M}}_n(t)$, at least one of $X_1(t)$ and $X_2(t)$ is non-zero. Therefore by equation (\ref{derofev}), there exists an open interval $I \subset \mathbb{R}$ containing $0$ such that $\lambda_k'(t) ~\textless~ 0$ for $t \in I$, i.e. $\lambda_k(t)$ is decreasing on $I$. This implies existence of $0 ~\textless~ t_0 ~\textless~ -b_2$ satisfying 
$$
\lambda_k ~\textless~ \lambda_k(t_0) \leq \lambda_k(-b_2) = \lambda_k(C_n) \leq \tilde{\lambda}_k.
$$
This contradicts with our assumption that $\lambda_k = \tilde{\lambda}_k$. Therefore only the first case of the claim is true, i.e. $b_1 = b_2 = 0$ and hence $\textnormal{\textbf{S}}_n = \textnormal{\textbf{F}}_n$.
\end{proof}

\section*{Acknowledgement}
The authors would like to thank Wencai Liu for introducing them this project and his constant support. This work was partially supported by NSF DMS-2015683 and DMS-2000345. 

\bibliography{main}

\end{document}